\documentclass[12pt]{article}
\usepackage{amsmath,amsfonts,amssymb}
\usepackage{alltt}
\usepackage{hyperref}
\usepackage{amsmath,amsfonts,ifthen,fullpage,enumerate,amsthm}

\usepackage{amsmath,amsfonts,amssymb}

\usepackage{makecell}  

\usepackage{graphicx}        

\usepackage{float}

\usepackage[table]{xcolor}   

\def\spam{\mathop{\rm span}\nolimits}

\def\Harm{\mathop{\rm Harm}\nolimits}

\def\Mol{\mathop{\rm Mol}\nolimits}
\def\hMol{\mathop{\rm Mol}^{\rm harm}\nolimits}
\def\hMolC#1{\mathop{M_{\Delta,#1}^\CC}\nolimits}

\def\ST{\mathop{\rm ST}\nolimits}
\def\SU{\mathop{\rm SU}\nolimits}
\def\GL{\mathop{\rm GL}\nolimits}

\def\tr{\mathop{\rm tr}\nolimits}

\def\vmat#1{\begin{vmatrix}#1\end{vmatrix}}

\def\pmat#1{\begin{pmatrix}#1\end{pmatrix}}

\def\question#1{{\bf Question: }#1}
\def\question#1{}

\def\cC{{\cal C}}

\def\R{\mathbb{R}}

\def\CC{\mathbb{C}}
\def\NN{\mathbb{N}}

\def\RR{\mathbb{R}}

\def\ZZ{\mathbb{Z}}

\def\Cd{\C^d}
\def\Rd{\R^d}

\def\C{\mathbb{C}}

\def\SS{\mathbb{S}}

\newtheorem{theorem}{Theorem}[section]
\newtheorem{corollary}[theorem]{Corollary}
\newtheorem{lemma}[theorem]{Lemma}
\newtheorem{example}[theorem]{Example}

\newtheorem{proposition}[theorem]{Proposition}

\newtheorem{conjecture}{Conjecture}

\newif\ifdraft\def\draft{\drafttrue\hoffset=.8truecm\showlabeltrue
\def\comment##1{{\bf comment: ##1}}
\headline={\sevenrm \hfill \ifx\filenamed\undefined\jobname\else\filenamed\fi%
(.tex) (as of \ifx\updated\undefined???\else\updated\fi)
 \TeX'ed at {\hour\time\divide\hour by 60{}%
\minutes\hour\multiply\minutes by 60{}%
\advance\time by -\minutes
\the\hour:\ifnum\time<10{}0\fi\the\time\  on \today\hfill}}
}

\chardef\other=12{}\chardef\active=13{}
\def\undospecials{\catcode`\\=\other \catcode`\{=\other
  \catcode`\}=\other \catcode`\$=\other \catcode`\&=\other
  \catcode`\#=\other \catcode`\%=\other \catcode`\~=\other
  \catcode`\_=\other \catcode`\^=\other \obeyspaces}
\def\|{\ifmmode\def\next{\delimiter"26B30D}%
\else\def\next{{\tt\char'174}}\fi\next}

{\obeylines\gdef\startdisplay#1
  {\catcode`\^^M=5$$#1\halign\bgroup\indent##\hfil&&\qquad##\hfil\cr}}
 \outer\def\enddisplay{\crcr\egroup$$}
\def\ttverbatim{\begingroup\undospecials\obeyspaces\obeylines%
\everypar{\strut\ }\tt}
{\obeyspaces\gdef {\ }} 
 \outer\def\begintt{$$\let\par=\endgraf \ttverbatim \parskip=0pt
  \catcode`\|=0 \rightskip=-5pc \ttfinish}
{\catcode`\|=0 |catcode`|\=\other 
  |obeylines 
  |gdef|ttfinish#1^^M#2\endtt{#1|vbox{#2}|endgroup$$}}
\newskip\ttglue\ttglue=.5em plus.25em minus.15em
\catcode`\|=\active
{\obeylines\gdef|{\ifmmode\def\next{\char`\|}\else%
  \def\next{\ttverbatim\spaceskip=\ttglue\let^^M=\ \let|%
  =\endgroup}\fi\next}}

\def\beginlines{\par\bgroup\nobreak\medskip\parindent=0pt
  \nobreak \obeylines \everypar{\strut}}
\def\endlines{\egroup\medbreak\noindent}

\def\inpro#1{\langle#1\rangle}
\def\ip#1{\langle\kern-.28em\langle#1\rangle\kern-.28em\rangle_\nu}

\def\cU{{\cal U}}

\def\norm#1{\Vert#1\Vert}
\def\openR{{{\rm I}\kern-.16em {\rm R}}}

\let\ga\alpha

\let\gb\beta

\let\gd\delta

\let\gep\varepsilon

\let\gl\lambda

\let\gs\sigma

\let\go\omega
\let\gw\omega
\let\gO\Omega
\let\ga\alpha

\let\gb\beta
\let\gd\delta
\let\gs\sigma
\def\inpro#1{\langle#1\rangle}

\def\Hom{\mathop{\rm Hom}\nolimits}

\def\Iff{\hskip1em\Longleftrightarrow\hskip1em}
\def\Implies{\hskip1em\Longrightarrow\hskip1em}

\def\mod{\rm\ mod\ }

\def\formeq{\the\sectionno.\the\equationno}  
\def\elabel#1/#2/#3/{\global\advance\equationno by 1 %
\ifx#1\empty\else\emember#1%
\ifshowlabel\marginal{\string#1}\fi\fi%
\ifmmode\eqno{#3(\formeq#2)}\else#3\formeq#2\fi} 

\def\makeblanksquare#1#2{
\dimen0=#1pt\advance\dimen0 by -#2pt
      \vrule height#1pt width#2pt depth0pt\kern-#2pt
      \vrule height#1pt width#1pt depth-\dimen0 \kern-#1pt
      \vrule height#2pt width#1pt depth0pt \kern-#2pt
      \vrule height#1pt width#2pt depth0pt
}

\title{\bf 
Complex spherical designs from group orbits
}

\author{
Mozhgan Mohammadpour, Shayne Waldron\\
 \\
Department of Mathematics \\ University of Auckland\\
Private
Bag 92019, Auckland, New Zealand\\
e-mail: waldron@math.auckland.ac.nz}
\date{\today}

\begin{document}

\maketitle 

\begin{abstract}
We consider the general question of when all orbits
under the unitary action of a finite group give a complex 
spherical design. Those orbits which have large stabilisers 
are then good candidates for being optimal complex spherical
designs.
This is done by developing the general theory of complex designs
and associated (harmonic) Molien series for group actions.
As an application, we give explicit constructions of some 
putatively optimal 
real and complex spherical 
$t$-designs.  

\end{abstract}

\bigskip
\vfill


 \ifdraft (Other designs, configurations) \else\fi
42C15, \ifdraft General harmonic expansions, frames  \else\fi
65D30; \ifdraft (Numerical integration) \else\fi
\quad
secondary
94A12. \ifdraft (Signal theory [characterization, reconstruction, etc.]) \else\fi

\vskip .5 truecm
\hrule
\newpage

\section{Introduction}

The notion of ``spherical designs'' and their interpretation as cubature
formulas for the real sphere were introduced by \cite{DGS77}.
The theory developed to include designs for the complex and quaternionic
spheres \cite{H90}, but these have been far less intensively studied.
Recently there has been renewed interest in complex spherical designs because
of their applications in quantum information theory \cite{RS14}.

Heuristically, good spherical designs, i.e., those with a small number of points, 
have some symmetries, and so are an orbit of some group.  
This is the main method for explicit constructions, and has been understood 
since Sobolev \cite{S62}.

Let $G$ be a finite abstract group with a unitary action on $\Cd$.
It is easy to show that if the action is irreducible, then every $G$-orbit of a unit vector is a tight
frame for $\Cd$, or, equivalently, is a complex spherical design
which integrates the polynomial space $\Hom(1,1)$, or
equivalently $H(1,1)$ 
 (see \S 10.5 \cite{W18}).
Here we investigate the question:

\medskip
{\it What class of complex spherical
designs is given by a general $G$-orbit?}
\medskip

\noindent
The corresponding results for orthogonal 
group actions on $\Rd$ have been well studied in \cite{GS81}, \cite{B84} and \cite{HP04}.

Here is an outline of our approach:
\begin{itemize} 
\item We require that our classes of complex spherical designs be invariant 
under unitary transformations. This includes all cases that have 
been considered in the literature. 
\item
Consequently, the space of polynomials that they integrate can be written as an 
orthogonal direct sum of the (absolutely)
irreducible unitarily invariant spaces of (harmonic) polynomials on the sphere 
$$ P_\tau:=\bigoplus_{(p,q)\in\tau} H(p,q), \qquad
\hbox{where $\tau\subset\NN^2$}. $$
A design which integrates $P_\tau$ will be called a (complex spherical) $\tau$-design.
\item By using a harmonic Molien series, we can calculate the largest $\tau=\tau_G$ 
for which a general $G$-orbit is a complex spherical $\tau$-design
(Theorem \ref{Gorbitdesigntype}).
\item We partition $\tau=\tau_P\cup\tau_S\cup\tau_E$ into projective indices $\tau_P$, i.e., those of the form $(p,p)$, which correspond to projective complex spherical designs, 
those $\tau_S$ which correspond to the scalar matrices 
which fix the design,
and $\tau_E$ the exceptions. 
\item
We show that multiplying the vectors in a design by an appropriate set of roots of unity
gives a $\tau$-design, where $\tau$ can contain any desired finite set of nonprojective indices
(Theorem \ref{rootsofunitykdesignTheorem}).
This gives a simple way to obtain spherical $t$-designs for $\Cd$ 
from complex projective spherical designs, 
and hence corresponding real spherical $t$-designs for $\RR^{2d}$.
By using this construction, we are able to obtain explicit expressions for various putatively 
optimal real and complex spherical $t$-designs.
\item We give extensive calculations of $\tau_G$ and its projective indices for various 
groups $G$ including the complex reflection groups, the Heisenberg and Clifford groups,
and various (canonical) representations in the special unitary group
(Section \ref{harmMolseriescalcsect}).
\item We show that the standard equations defining a Weyl-Heisenberg SIC follow from viewing 
it as an orbit of the Heisenberg group which is a $(2,2)$-design
(Section \ref{SICsection}).
\end{itemize}
We now formalise the concepts discussed above, motivated by the development
of complex spherical designs given by \cite{RS14}.

Let $\SS$ be the unit sphere in $\Rd$ or $\Cd$, 
and $\gs$ be the normalised surface area measure on $\SS$.
A {\bf spherical design} (for $P$) is a finite set (or sequence, or multiset) points $X$ in $\SS$
for which the integration (cubature) rule
\begin{equation}
\label{cuberule}
\int_\SS p(x)\, d\gs(x) = {1\over |X|} \sum_{x\in X} p(x),
\end{equation}
holds for all $p$ in a 
subspace of polynomials $P$.
We say that a polynomial $p$ (or a space of polynomials) is
{\bf integrated} by the spherical design 
if (\ref{cuberule}) holds.
The common choices for $P$ are unitarily invariant,
i.e.,  for $U$ unitary, $p\circ U\in P$, $\forall p\in P$.
In the real case, the unitary maps are the orthogonal transformations.
For such a space $P$, the unitary invariance of the measure $\gs$ implies
that $UX=(Ux)_{x\in X}$ is also a spherical design when $U$ is  unitary,
by the calculation
$$ {1\over|X|}\sum_{x\in X} p(Ux)  
=\int_\SS p(Ux)\, d\gs(x)
=\int_\SS p(x)\, d\gs(x), \qquad \forall p\in P. $$

Every unitarily invariant space of polynomials $P$ on the complex sphere can be written
uniquely as an orthogonal direct sum
\begin{equation}
\label{Ptaudefn}
P=P_\tau=\bigoplus_{(p,q)\in\tau} H(p,q),  \qquad \tau\subset\NN^2,
\end{equation}
with respect to the inner product
$$ \inpro{f_1,f_2} := \int_\SS f_1(z)\overline{f_2(z)}\, d\gs(z), $$
where 
$H(p,q)$ is the (absolutely) irreducible unitarily invariant
subspace of the harmonic homogeneous polynomials on the complex sphere in $\Cd$ 
which are of degree $p$ in the variables $z=(z_1,\ldots,z_d)$ and of
degree $q$ in $\overline{z}=(\overline{z_1},\ldots,\overline{z_d})$.  
The orthogonality gives
\begin{equation}
\label{Hpqintformula}
H(0,0)=\spam\{1\}, \qquad
\int_\SS f\, d\gs = 0, \quad\forall f\in H(p,q),\ \forall (p,q)\ne(0,0).
\end{equation}
These have been well studied \cite{R80}. Note, in particular
(see Example \ref{Hpqdimremark}), that
$$\dim(H(p,q)) 
= {d+p-1\choose d-1} {d+q-1\choose d-1} -{d+p-2\choose d-1} {d+q-2\choose d-1}, $$
and
\begin{equation}
\label{Hompqexpansion}
\Hom(p,q)|_\SS=H(p,q)\oplus H(p-1,q-1)\oplus H(p-2,q-2)\oplus \cdots, 
\end{equation}
where $\Hom(p,q)$ is the subspace of homogeneous polynomials of degree $p$ in $z$
and degree $q$ in $\overline{z}$. 
The subspaces $H(p,q)$ are all nonzero, except in the
(degenerate) case $d=1$, where they are zero, except for 
\begin{equation}
\label{Hpqdegeneratecase}
H(p,0)=\spam\{z^p\}, \qquad H(0,q)=\spam\{\overline{z}^q\}.
\end{equation}

A design $X$ which integrates $P_\tau$ of (\ref{Ptaudefn}) 
is called a {\bf complex spherical $\tau$-design},
or simply a {\bf $\tau$-design}. We will
also say $X$ {\bf integrates} the indices $\tau$. 
We denote by $\tau_X$ the {\bf maximal} possible $\tau$ 
(in the sense of set inclusion) for $X$. 
These designs have been studied in \cite{RS14} for $\tau$
a {\bf lower set} in $\NN^2$, i.e., one with the property that 
$$ (p,q)\in\tau \Implies \{(k,l):0\le k\le p,0\le l\le q\}\subset\tau. $$
There a design is said to be 
{\bf $(k,l)$-regular} if it integrates $\Hom(k,l)$, i.e., 
by (\ref{Hompqexpansion}),
$$ \{(k,l),(k-1,l-1),(k-2,l-2),\ldots\}\subset \tau_X. $$
We note that 
\begin{itemize}
\item $(0,0)\in\tau_X$.
\item $\tau_X$ is symmetric, i.e., if $(p,q)\in\tau_X$
then $(q,p)\in\tau_X$. 
\end{itemize}
The latter property follows from the fact that elements of
$H(p,q)$ can be written as real linear combinations of monomials $z^\ga\overline{z}^\gb$, $|\ga|=p$, $|\gb|=q$.

We recall the classical definition that $X$ is a {\bf $t$-design} if it integrates 
all polynomials of degree $\le t$, i.e., in complex case
\begin{equation}
\label{classicaltdesigndef}
\tau_t:=\{(p,q)\in\NN^2:p+q\le t\}\subset\tau_X.
\end{equation}
We observe that a complex $t$-design for $\Cd$ corresponds to a 
real $t$-design for $\RR^{2d}$.
Further, we say that $X$ is a spherical {\bf $(t,t)$-design} if it is $(t,t)$-regular, i.e., 
$$ \{(0,0),(1,1),\ldots,(t,t)\} \subset \tau_X. $$
Equivalently, by (\ref{Hompqexpansion}), it integrates $\Hom(t,t)$.
These complex spherical $(t,t)$-designs, and their real counterparts, 
are ``projective'' designs, 
which are well studied \cite{H92}, \cite{KP17}, \cite{W18} (Section 6).

The real analogue of a complex spherical $\tau$-design was
introduced in \cite{M13}, \cite{ZBBE17} as the class of
{\bf spherical designs of harmonic index $T$} (or {\bf $T$-designs} for short), where $T\subset\NN$,
and the space integrated is $P_T=\bigoplus_{k\in T}\Harm(k)$,
with $\Harm(k)$ the 
harmonic homogeneous polynomials
of degree $k$. Clearly $T=\{0,1,\ldots,t\}$ gives the real spherical $t$-designs, 
and a complex $\tau$-design corresponds to a real spherical $T$-design if and only if
$$ \tau=\{(p,q):p+q=k,k\in T\}. $$

A ``good'' $\tau$-design is one with a {\em small} number of points which integrates
a {\em large} space of polynomials. We now formalise this idea.
Let $X$ and $Y$ be designs (finite subsets of $\SS$).
We say that $X$ is {\bf better} than $Y$ if
$$ |X|\le|Y| \quad\hbox{($X$ has fewer points)}, \qquad\tau_X\supset\tau_Y
\quad\hbox{($X$ integrates more polynomials)}, $$
and {\bf strictly better} if there is a strict inequality or inclusion above.
This gives a preorder (quasi-order) on designs, 
and a partial order on the pairs $(|X|,\tau_X)$.
Hence, we obtain a partial order on designs
if we identify those with $|X|=|Y|$ and $\tau_X=\tau_Y$.
We will say that a spherical design is {\bf optimal} if it is 
the maximal element in this poset.
The usual notion of optimality for spherical designs considers only the number of points.

\begin{example} An optimal spherical $\tau$-design $X$ is one with
the minimal number of points. Such a design may not be optimal
in the class of spherical designs if there is another one $Y$ with
$|X|=|Y|$ and $\tau\subset\tau_X\subsetneq\tau_Y$.
\end{example}

\section{The design given by a group orbit}

From now on, unless stated otherwise, 
let $G$ be a finite abstract group with a unitary action on $\Cd$,
i.e., a group homomorphism $\rho:G\to\cU(\Cd)$ to the unitary matrices,
with $gv:=\rho(g)v$, $v\in\Cd$. We will refer to $\rho(G)$ as the 
{\bf action group} of $G$. This induces a natural group action on
functions $f:\Cd\to\CC$ given by
$$ (g\cdot f)(v) = f(g^{-1}v), \qquad v\in\Cd. $$
The {\bf $G$-invariant subspace} 
of a space $V$ of such functions, e.g., 
$V=H(p,q)$, is 
$$ V^G := \{ f\in V: g\cdot f =f, \forall g\in G\}. $$
This subspace is the image of the {\bf Reynolds operator} $R_G:V\to V$
given by 
\begin{equation}
\label{Reynoldsopdefn}
R_G(f):= {1\over|G|}\sum_{g\in G} g\cdot f. 
\end{equation}
Since the surface area measure is unitarily invariant, we have
\begin{equation}
\label{Reynoldsopintegral}
\int_\SS f \,d\gs = \int_\SS R_G(f)\,d\gs.
\end{equation}

The following key result determines the type of design for
a $G$-orbit 
$$Gv:=(gv)_{g\in G}, \qquad v\in\SS.$$  

\begin{lemma}
\label{keylemma}
Let $G$ be a finite abstract group with a unitary action on $\Cd$.
Then a $G$-orbit $(gv)_{g\in G}$, $v\in\SS$, integrates $H(p,q)$ if and only if
$(p,q)=(0,0)$ or
\begin{equation}
\label{G-orbitintegrates}
 f(v)=0, \qquad\forall f\in H(p,q)^G. 
\end{equation}
In particular, if $H(p,q)^G=0$, $(p,q)\ne(0,0)$, then every $G$-orbit integrates $H(p,q)$.
\end{lemma}

\begin{proof}
Note that $R_G(H(p,q))=H(p,q)^G$.
 The $G$-orbit $(gv)_{g\in G}$ integrates $H(p,q)$ if and only if
$$ {1\over|G|} \sum_{g\in G} f(gv)=\int_\SS f(x)\,d\gs(x), \qquad \forall f\in H(p,q), $$
which, by (\ref{Reynoldsopdefn}) and (\ref{Reynoldsopintegral}), we can write as
$$ (R_G(f))(v) = \int_\SS R_G (f)\,d\gs, \qquad \forall f\in H(p,q), $$
i.e., 
$$ f(v)=\int_\SS f\, d\gs, \qquad\forall f\in H(p,q)^G.$$
This is satisfied for $H(0,0)=\spam\{1\}$ (the constants), and
for $(p,q)\ne0$, by (\ref{Hpqintformula}),
it reduces to  (\ref{G-orbitintegrates}). 
\end{proof}

Choosing a basis (or spanning set) $\{f_j\}$ for $H(p,q)^G$, 
(\ref{G-orbitintegrates})
gives a system $f_j(v)=0$ of polynomial equations 
for the orbit of $v$ to integrate $(p,q)$. In Section \ref{SICsection},
this is considered in detail for
when an orbit of the Weyl-Heisenberg group is a spherical $(2,2)$-design.

\begin{theorem}
\label{Gorbitdesigntype}
Let $G$ be a finite group with a unitary action on $\Cd$,
and $\tau\subset\NN^2$ be a set of indices.
Then every $G$-orbit 
is a spherical $\tau$-design if and only if
$$ \tau\subset \tau_G, \qquad \tau_G := \{(0,0)\} \cup 
 \{(p,q): H(p,q)^G=0\}, $$
i.e.,
$\dim(H(p,q)^G) = 0$, $\forall (p,q)\in\tau$, $(p,q)\ne(0,0)$.
\end{theorem}

\begin{proof} Observe that each $G$-orbit is a $\tau$-design, i.e., it integrates
$P_\tau$ of (\ref{Ptaudefn}) if and only if it integrates 
each $H(p,q)$, $(p,q)\in\tau$, and apply Lemma \ref{keylemma}.
\end{proof}

This result is given for $\tau$ a lower set in \cite{RS14} (Theorem 11.1).

\begin{example}
For a particular orbit $X=Gv=(gv)_{g\in G}$, the maximal $\tau$ 
is possibly larger than $\tau_G$, and is given by
\begin{align*}
\tau_X 
= \tau_{Gv} &=\{(0,0)\}\cup\{(p,q):f(v)=0,\forall f\in H(p,q)^G\} \cr
&= \tau_G \cup \{(p,q)\not\in\tau_G:f(v)=0,\forall f\in H(p,q)^G\}
\quad \hbox{\rm (disjoint union)}.
\end{align*}
\end{example}

In light of these results, natural ways to obtain good spherical 
designs as group orbits $Gv$ include

\begin{itemize}
\item Minimize the cardinality of the set $Gv$ over $v\in\SS$,
i.e., have $v$ fixed by a large subgroup (see \cite{BW13}).
\item Maximize the cardinality of $\tau_{Gv}\setminus\tau_G$ over all points $v\in\SS$.
\item Take a union of two or more orbits (see \cite{MW19}).
\end{itemize}

All of this is of course dependent on a practical method for calculating
$H(p,q)^G$, or a least when it is zero. Such a method is provided by 
a harmonic Molien series.

\section{Harmonic Molien series}

To apply Theorem \ref{Gorbitdesigntype}, it suffices to know the dimensions of the $H(p,q)^G$. These
can be computed as the coefficients of the 
{\it harmonic Molien series} of \cite{RS14} (Corollary 11.7)
\begin{equation}
\label{complexharmMolseries}
\hMol_G(x,y) := \sum_{(p,q)} \dim(H(p,q)^G) x^py^q
= {1\over|G|} \sum_{g\in G} {1-xy\over \det(I-x g) \det(I-y g^{-1}) }.
\end{equation}
This holds for all linear actions of $G$. 
If the action is unitary, then $g^{-1}$ can be replaced by 
$\overline{g}$ (as in the original presentation) or by $g^*$.

\begin{example} 
\label{Hpqdimremark}
For the identity group $G=1$, 
the harmonic Molien series is
$$ \hMol_{1}(x,y) 
= {1-xy\over(1-x)^d(1-y)^d}
= (1-xy) 
\Bigl(\sum_{a=0}^\infty {d+a-1\choose d-1} x^a\Bigr)
\Bigl(\sum_{b=0}^\infty {d+b-1\choose d-1} y^b\Bigr).
$$ 
Since $H(p,q)^1=H(p,q)$, the $x^py^q$ coefficient gives
\begin{align*}
\dim(H(p,q)) 
&= {d+p-1\choose d-1} {d+q-1\choose d-1} -{d+p-2\choose d-1} {d+q-2\choose d-1} \cr
&= {(d+p-2)!\over p!(d-1)!} {(d+q-2)!\over q!(d-1)!} (d-1)(d-1+p+q), \quad d\ge2.
\end{align*}
\end{example}

The formula presented in (\ref{complexharmMolseries}) 
is invariant under a similarity transformation,
and so holds for an equivalent (possibly nonunitary) representation.
Since each term of (\ref{complexharmMolseries}) depends only on $g$ up
to similarity, and hence conjugacy in $G$, the harmonic Molien series
can be calculated from the conjugacy classes $\cC$ of $G$, via
\begin{equation}
\label{complexharmMolseriesconjclasses}
\hMol_G(x,y) = {1\over|G|} \sum_{C=[g]\in \cC} |C|
{1-xy\over \det(I-x g) \det(I-y g^{-1}) }.
\end{equation}

The formula (\ref{complexharmMolseriesconjclasses}) is very practical 
for computations (see Section \ref{harmMolseriescalcsect}).

\begin{example}
\label{harmMolmagmacode}
For $G$ the Shephard-Todd reflection group number $37$, i.e.,
the Weyl group $W(E_8)$ acting on $\CC^8$, we have
$|G|=696729600$, $|\cC|=112$, and the {\tt magma} code
\begin{verbatim}
G:=ShephardTodd(37);
P<x>:=PolynomialRing(BaseRing(G)); Q<y>:=PolynomialRing(P);
d:=Dimension(G); MR:=MatrixRing(Q,d); I:=IdentityMatrix(Q,d);
CC:=ConjugacyClasses(G);

sm:=0;
for j in [1..#CC] do
  g:=CC[j][3];
  sm:=sm+CC[j][2]*(1-x*y)/Determinant(I-x*MR!g)/Determinant(I-y*MR!(g^-1));
end for;
ms:=sm/Order(G);
\end{verbatim}
calculates the harmonic Molien series from the conjugacy classes in 
under three seconds,
whilst direct calculation of (\ref{complexharmMolseries}) is unfeasible.
\end{example}

In view of (\ref{complexharmMolseriesconjclasses}), 
the condition that each orbit integrates $H(p,q)$ can be written as
\begin{equation}
\label{harmMoldcharpolyform}
\dim(H(p,q)^G)={1\over p!q!}{\partial^p\over\partial x}{\partial^q\over\partial y}
{1\over|G|} \sum_{g\in G} {1-xy\over \det(I-x g) \det(I-y g^{-1}) } 
\Big\vert_{(x,y)=(0,0)} = 0.
\end{equation}
To calculate this, we observe that the factor
$\det(I-x g)$ in the denominator
is a multiple of the characteristic polynomial of $g^{-1}$,
which has degree $d$, 
and can be expanded in terms of the traces of the exterior powers of $g$
(see \cite{Wik23}),
i.e.,
\begin{equation}
\label{f_gTaylorpoly}
f_g(x):=\det(I-xg) = \sum_{k=0}^d (-1)^k \tr(\wedge^k g) x^k, 
\end{equation}
where
\begin{equation}
\label{TraceExteriorPowers}
\tr(\wedge^k g) = {1\over k!}
\vmat{
\tr(g) & k-1 & 0 & \cdots & 0 \cr
\tr(g^2) & \tr(g) & k-2 & \cdots & 0 \cr
\vdots & \vdots & & \ddots & 0 \cr
\tr(g^{k-1}) & \tr(g^{k-2}) & & \cdots & 1 \cr
\tr(g^{k}) & \tr(g^{k-1}) & & \cdots & \tr(g) 
}.
\end{equation}
This gives the following.

\begin{proposition}
\label{H(p,q)^Gdimensionformula}
Let $G$ be a finite group with a linear action on $\Cd$.
Then 
\begin{equation}
\label{harmMoldcharpolyformulaII}
\dim(H(p,q)^G)
= \dim(H(q,p)^G)
= {1\over p!q!}{1\over|G|} \sum_{g\in G} 
{\partial^p\over\partial x}{\partial^q\over\partial y}
{1-xy\over f_g(x)f_{g^{-1}}(y)} \Big\vert_{(x,y)=(0,0)}, 
\end{equation}
where $f_g(x)=\det(I-xg)$, and 
\begin{equation}
\label{f_gvalsat0}
f_g^{(k)}(0)=
\begin{cases}
1, & k=0; \cr
(-1)^k k! \tr(\wedge^k g), & 1\le k\le d; \cr
0, & k>d.
\end{cases}
\end{equation}
In particular,
\begin{align*}
\dim(H(1,0)^G) &= {1\over|G|} \sum_{g\in G} \tr(g), \cr
\dim(H(2,0)^G) &= {1\over2} {1\over|G|} \sum_{g\in G} \bigl\{ \tr(g)^2+\tr(g^2) \bigr\},\cr
\dim(H(3,0)^G) &= {1\over6} {1\over|G|} \sum_{g\in G} \bigl\{\tr(g)^3+2\tr(g^3)+3\tr(g)\tr(g^2)\bigr\}, \cr
\dim(H(1,1)^G) & = {1\over|G|}\sum_{g\in G} \bigl\{\tr(g)\tr(g^{-1})-1\}, \cr
\dim(H(2,1)^G) &= {1\over2} {1\over|G|} \sum_{g\in G} \bigl\{\tr(g)^2{\tr(g^{-1})} 
+\tr(g^2){\tr(g^{-1})}-2\tr(g)\bigr\}.
\end{align*}
Moreover, since each term above depends only on the eigenvalues of $g$, the sums
can be taken over conjugacy classes, if so desired.
\end{proposition}

\begin{proof} The formula (\ref{harmMoldcharpolyformulaII}) follows by equating
coefficients of (\ref{complexharmMolseries}), and (\ref{f_gvalsat0}) 
from the Taylor polynomial (\ref{f_gTaylorpoly}). 
The particular cases  are by direct calculation, 
using (\ref{TraceExteriorPowers}) 
to calculate $\tr(\wedge^k g)$, e.g., 
$$ f_g'(0)=-\tr(g), \quad f_g''(0)=\tr(g)^2-\tr(g^2),
\quad f_g^{(3)}(0)=-\tr(g)^3+3\tr(g)\tr(g^2)-2\tr(g^3). $$
For example, when $(p,q)=(1,0)$, we have
$$ {\partial\over\partial x} {1-xy\over f_g(x)f_{g^{-1}}(y)} \Big\vert_{(0,0)} 
= {f_g(x)f_{g^{-1}}(y)(-y)-(1-xy)f_g'(x)f_{g^{-1}}(y)\over 
(f_g(x)f_{g^{-1}}(y))^2} \Big\vert_{(0,0)}
= -f_g'(0) = \tr(g), $$
which gives the formula for $\dim(H(1,0)^G)$.
\end{proof}


\begin{example} 
\label{H(1,1)irreduciblecdn}
The condition for $H(1,1)$ to be integrated by every orbit
can be written
\begin{equation}
\label{SerreCondition}
\sum_{g\in G} \tr(g)\tr(g^{-1}) =|G|, 
\end{equation}
and is equivalent to: 
	(i) each orbit being a tight frame \cite[Example 1.3]{HW21},
(ii) the action of $G$ being irreducible 
	\cite[\S 10.5]{W18}.
The condition (\ref{SerreCondition}) for a linear action to
	be irreducible is given in  \cite[Theorem 5]{S77}. 
\end{example}

It is also possible to give an explicit formula for $\dim(H(p,q)^G)$ in terms of the
eigenvalues $\gl_g=(\gl_{1,g},\ldots,\gl_{d,g})$ of $g$. Since $\det(I-xg)=\prod_j(1-\gl_{j,g}x)$,
by using geometric series expansions,
and equating coefficients in (\ref{complexharmMolseries}), we obtain the formula
\begin{equation}
\label{HpqGineigenvalues}
\dim\bigl(H(p,q)^G\bigr)= 
{1\over|G|}\sum_{g\in G} \Bigl\{ 
\sum_{|\ga|=p \atop |\gb|=q} \gl_g^\ga\overline{\gl_g}^\gb 
-\sum_{|\ga|=p-1 \atop |\gb|=q-1} \gl_g^\ga\overline{\gl_g}^\gb \Bigl\},
\qquad p,q\ge1.
\end{equation}
Since each $g$ has finite order, these are sums of roots of unity 
(see Example \ref{binarydihedexample}).

Let $V^K$ denote the subspace of $V$ invariant under the action
of a subset $K$ of $G$. It is possible to calculate $G$-invariant subspaces
from those invariant under suitable subsets.

\begin{proposition}
Let $K_j$ be subsets of $G$ for which $\cup_j K_j$ generates $G$.
Then 
$$ H(p,q)^G = \bigcap_j H(p,q)^{K_j}, $$
and, in particular, for subsets $K$ and $L$ of $G$, we have
$$ K\subset L \Implies H(p,q)^L \subset H(p,q)^K. $$
Thus for $K$ and $L$ subgroups, we have
\begin{equation}
\label{KLtau}
K\subset L \Implies \tau_K \subset \tau_L,
\end{equation}
i.e., larger groups integrate more polynomials.
\end{proposition}

\section{The role of scalar matrices and projective indices}

In view of (\ref{KLtau}), we can determine some indices in $\tau_G$,
by considering subgroups. We now consider the subgroup of scalar
matrices, which plays an important role in the theory of 
complex (and real) spherical designs.

\begin{lemma}
\label{roleofscalarslemma}
Let   
$\gw:=e^{2\pi i/k}$, a primitive $k$-th root of unity,
and $\inpro{\gw I}$ be the group of scalar matrices acting on $\Cd$ that it generates. 
Then 
$$ H(p,q)^{\inpro{\gw I}}
=\begin{cases}
0,& p-q\not\equiv 0\mod k; \cr
H(p,q),& p-q\equiv 0\mod k.
\end{cases} $$
In particular, if the group of scalar matrices in the action group of $G$ has order $k$,
then every $G$-orbit is a spherical design for the indices 
\begin{equation}
\label{taukdefn}
\tau_k^S := \{(p,q) : p-q \not\equiv 0 \mod k\}\subset\tau_G.
\end{equation}
\end{lemma}


\begin{proof} This follows by expanding the harmonic Molien series using
binomial series. However, we give an elementary argument based 
on the Reynolds operator (\ref{Reynoldsopdefn}). 
Let 
$$ z^\ga\overline{z}^\gb\in\Hom(p,q)=H(p,q)\oplus H(p-1,q-1)\oplus\cdots. $$
Then
\begin{align*}
R_{\inpro{\go I}} (z^\ga\overline{z}^\gb)
& = {1\over k}\bigl( z^\ga\overline{z}^\gb+ (\go z)^\ga(\overline{\go z})^\gb+\cdots
+ (\go^{k-1} z)^\ga(\overline{\go^{k-1} z})^\gb\bigr) \cr
&= {1\over k}\bigl(1+\go^{p-q}+\cdots + \go^{(p-q)(k-1)}\bigr) z^\ga\overline{z}^\gb \cr
&=\begin{cases}
0,& p-q\not\equiv 0\mod k; \cr
z^\ga\overline{z}^\gb, & p-q\equiv 0\mod k,
\end{cases}
\end{align*}
so that
$$ \Hom(p,q)^{\inpro{\go I}}
=\begin{cases}
0,& p-q\not\equiv 0\mod k; \cr
\Hom(p,q),& p-q\equiv 0\mod k,
\end{cases}
$$
which gives the result since $H(p,q)\subset\Hom(p,q)$.
\end{proof}

The following is a key tool for constructing and analysing complex
spherical designs.

\begin{corollary}
\label{worbitcor} 
Let $\gw=e^{2\pi i/k}$. For any $v\in\SS$, the set
$$ \{v,\go v,\go^2 v,\ldots,\go^{k-1} v\},$$
or union of such sets, is a spherical $\tau_k^S$-design.
\end{corollary}

\begin{proof} First note that a union of $\tau$-designs is a $\tau$-design.
Observe that the given set is the orbit of $v$ under the
action of the scalar matrices $\inpro{\gw I}$,
and apply Lemma \ref{roleofscalarslemma}.
\end{proof}

This result is effectively Lemma 2.5 of \cite{RS14}.
There $X$ is said to be a {\bf $k$-antipodal} if it can be partitioned
$(L,\gw L,\gw^2 L,\cdots,\gw^{k-1}L)$, i.e., is a union of $\inpro{\gw I}$-orbits, 
and $X$ is said to be a {\bf $k$-antipodal cover of $L$}.


We will say that a spherical design $(v_j)$
is a {\bf projective design} if $(c_jv_j)$ is a spherical design
for all choices of unit scalars $c_j$. 
Clearly, such a design can be thought of as a sequence of lines.

Let $\gw=e^{2\pi i/k}$. It follows from Corollary \ref{worbitcor} that
\begin{itemize}
\item If $X$ is any design, then $\{1,\go,\ldots,\go^{k-1}\}X$ is $k$-antipodal, 
and hence is a $\tau_k^S$-design.
\end{itemize}
Thus the only indices that can't be integrated by making a design $k$-antipodal
in this way (for $k$ sufficiently large) are 
$$ \{(1,1),(2,2),(3,3),\ldots\}, $$
which, together with $(0,0)$, we call the {\bf projective indices}.
The reason for this is that the designs which integrate
$\{(0,0),(1,1),\ldots,(t,t)\}$, called spherical 
$(t,t)$-designs in \cite{W18}, are projective designs.
This is easily seen from their characterisation
\begin{equation}
\label{varcharofdesigns}
\sum_{x,y\in X}|\inpro{x,y}|^{2t} = 
{1\over{t+d-1\choose t}} 
\Bigl(\sum_{z\in X}\norm{z}^{2t}\Bigr)^2. 
\end{equation}
These designs are also developed in \cite{H90} in a projective setting.



\section{Constructing spherical $t$-designs from projective designs}

We now give a precise statement of our observation that multiplying 
a projective design by appropriate roots of unity gives a spherical $t$-design.
Let $\go_k:=e^{2\pi i/k}$.


\begin{theorem}
\label{rootsofunitykdesignTheorem}
Let $(v_j)$ be a spherical $(t,t)$-design of $n$ vectors 
for $\Cd$,
and $0\le k\le 2t+1$. Then
\begin{equation}
\label{rootsofunitykdesign}
(\go_{k+1}^a v_j)_{0\le a\le k,1\le j\le n} 
\end{equation}
is a spherical $k$-design of $(k+1)n$ vectors for $\Cd$, and hence
for $\RR^{2d}$.
\end{theorem}

\begin{proof} Let $X=(v_j)$. We recall (\ref{classicaltdesigndef}),
 that a complex spherical $t$-design is one which integrates the indices
$$ \tau_t:=\{(p,q)\in\NN^2:p+q\le t\}. $$
Since $X$ is a $(t,t)$-design, it integrates all the projective
indices in $\tau_{2t+1}$. Multiplying any design by the $(k+1)$-th roots 
of unity gives a design which integrates all the nonprojective 
indices in $\tau_{k}$. 
Thus, by Corollary \ref{worbitcor}, the design $\{1,\go,\ldots,\go^k\}X$, $\go:=\go_{k+1}$, of (\ref{rootsofunitykdesign})
is a $\tau$-design with
$$ \tau_{k}\subset\{(0,0),\ldots,(t,t)\}\cup\tau_{k+1}^S\subset \tau, $$
i.e., is a spherical $k$-design for $\Cd$,
and hence for $\RR^{2d}$.
\end{proof}

This result is the complex analogue of the result that if $X$ is a real design 
which integrates the even polynomials of
degree $\le 2t$ (which is a projective design), 
equivalently 
$$ \Hom(2t)=\Harm(0)\oplus\Harm(2)\oplus\cdots\oplus\Harm(2t-2)\oplus\Harm(2t), $$
then the centrally symmetric set $\{1,-1\}X=X\cup -X$ is a real spherical 
$(2t+1)$-design.

There are obvious variations of Theorem \ref{rootsofunitykdesignTheorem}.
For example, the projective design may already integrate some (or all) of the
indices in $\tau_{k+1}^S$, and so a smaller value of $k$ might suffice
to construct a design of desired strength, or $X$ may already be $k$-antipodal 
for some $k$, and so the constructed design may have fewer points (repeated
vectors) in this case.

We now give examples, starting with the degenerate case $d=1$.

\begin{example} For $d=1$, the only indices 
which correspond to nontrivial subspaces $H(p,q)$, see 
(\ref{Hpqdegeneratecase}), are $(0,0)$ and the nonprojective indices
$$  (1,0),(0,1),\ (2,0),(0,2),\ (3,0),(0,3),\ldots $$
and so every subset of $\CC$ is a spherical $(t,t)$-design, for all $t$.
Hence we can take a single point $\{v\}$ ($n=1$) 
multiplied by the $(k+1)$-th roots of unity
to obtain 
a spherical $k$-design $\{v,\go v,\ldots\go^k v\}$, $\go=\go_{k+1}$, 
for $\CC$. The corresponding real $k$-design for $\RR^2$ 
consists of $k+1$ equally spaced vectors (a regular polygon) \cite{H82}.
\end{example}

We now consider the case $d=2$, which explains the putatively optimal spherical
$t$-designs for $\RR^4$ found numerically by Sloane, et al 
\cite{CHS03}, which were observed to have the following structure:

\medskip
{\it ``we take a “nice” set of $n$ planes in $\RR^4$,
draw a regular (k+1)-gon in each plane for a given value of $k$, 
obtaining a set of $(k+1)n$ points that will form a spherical
$t$-design for some $t$ (possibly 0).''}
\medskip

\begin{example} 
\label{SICexample}
There is a unique set of $4$ equiangular lines in $\CC^2$,
which give a spherical $(2,2)$-design, which is much loved, and called a SIC
	(see \S \ref{SICsection}).
Multiplying by the $(k+1)$-th roots of unity gives a
real spherical $k$-design of  $4(k+1)$ vectors for $\RR^4$,
$1\le k\le 5$. 
\end{example}

\begin{example} 
There is a unique set of $6$ lines in $\CC^2$ which give
a spherical $(3,3)$-design. These are 
three MUBs (mutually unbiased bases) 
\cite{W18}.
Multiplying by the $(k+1)$-th roots of unity gives a
real spherical $k$-design of  $6(k+1)$ vectors for $\RR^4$,
$1\le k\le 7$. 
\end{example}

\begin{example} Consider the $12$ lines in $\CC^2$ which give a spherical 
$(5,5)$-design (see \cite{HW21} Example 4.2). 
Multiplying by the $(k+1)$-th roots of unity gives a
real spherical $k$-design of  $12(k+1)$ vectors for $\RR^4$,
$1\le k\le 11$. 
\end{example}

The above three examples are Theorems 1,2,3 of \cite{CHS03}.
Proofs are not provided, but the following description of a direct
verification is given:

\medskip
{\it ``Theorems 1--3 are established by computing the distance
distribution of the design and working out its Gegenbauer
transform. This is simplified by the fact that the planes
in the three theorems form isoclinic sets.'' }
\medskip

Since there is no spherical $(4,4)$-design for $\CC^2$ with less than $12$ points, 
there are no corresponding real designs of interest. However, we do have 
the following.

\begin{example}
There is a spherical $(7,7)$-design of $24$ points in $\CC^2$ \cite{HW21}, 
which gives a spherical $k$-design of $24(k+1)$ vectors for $\RR^4$, $1\le k\le15$.
\end{example}

We now present a couple of infinite families of real spherical $t$-designs.

\begin{example} An orthonormal basis gives a spherical $(1,1)$-design (tight frame)
for $\Cd$. Hence, multiplying this by the third and fourth roots of unity gives
\begin{itemize}
\item There is a $2$-design of $3d$ vectors for $\RR^{2d}$, 
which is not centrally symmetric.
\item There is a $3$-design of $4d$ vectors for $\RR^{2d}$, which is 
centrally symmetric.
\end{itemize}
For $d=2$, these designs have $6$ and $8$ vectors, and are the putatively
optimal spherical $t$-designs for $\RR^4$ \cite{CHS03}. 
The investigation of \cite{B98} into real spherical $3$-designs for $\RR^d$
concluded that for $\RR^{2d}$ (their $S^{2d-1}$) the minimal
number of vectors for a $3$-design
is $2(2d-1)+2=4d$, which is given by the vertices of the
``generalised regular octahedra'' (Construction 3.1).
This is the same as our construction (for $d$ even).
\end{example}

We now generalise Example \ref{SICexample} to Weyl-Heisenberg SICs 
(see Section \ref{SICsection}).

\begin{example}
Zauner \cite{Z11} conjectures that there is a set of $d^2$ 
equiangular lines in $\Cd$, also known as a SIC,
in every dimension $d$. A SIC is a spherical $(2,2)$-design.
This is a major question in quantum
information theory, and has been proved, by explicit construction, 
for many dimensions $d$
\cite{GS17}, \cite{ACTFW18}.
Given that a SIC exists, then multiplying it by the 
fifth and six roots of unity gives
\begin{itemize}
\item There is a $4$-design of $5d^2$ vectors for $\RR^{2d}$,
which is not centrally symmetric.
\item There is a $5$-design of $6d^2$ vectors for for $\RR^{2d}$, which is
centrally symmetric.
\end{itemize}
For $d=2$ (Example \ref{SICexample}) these designs
of $20$ and $24$ points for $\RR^4$ are optimal \cite{CHS03}.
An explicit construction of $5$-designs of $n$ points in $\RR^{2d}$
is given in \cite{B91} (Theorem 1), where
$$ n > \max\{ 2^{2d}(d+1)/d+8(d+1)(4d^2+1)/(2d+1),16(2d-1)(2d^2+d+1)\}>6d^2. $$
The above $5$-designs given by SICs improve upon this.
\end{example}

Table \ref{RealSphereDesignTable} summarises the above examples of
Theorem \ref{rootsofunitykdesignTheorem}.

\setlength{\tabcolsep}{3pt}
\renewcommand{\arraystretch}{1.15}
\begin{table}[H]
\fontsize{10pt}{10pt}\selectfont
\caption{\small Examples of spherical $k$-designs 
for $\Cd$ and $\RR^{2d}$ constructed by Theorem \ref{rootsofunitykdesignTheorem}.
Those with an ${}^*$ are conjectured by \cite{CHS03} 
to give optimal real spherical $k$-designs 
for $\RR^4$.  
Note, that for $k$ odd, these designs are centrally symmetric, and otherwise are not.}
\begin{center}
\label{RealSphereDesignTable}
\begin{tabular}{|p{1.6cm}| p{1.3cm}|p{1.4cm}|p{6.0cm}|}
\hline
complex dimension & $k$-design strength & number of points & projective design/comment \\ 
$d$ & $k$ & $(k+1)n$ & \\ \hline
$d$ & 2 & $3d$ & ONB (orthonormal basis) \\
$d$ & 3 & $4d$ & ONB, optimal (generalised octahedra) \\
$d$ & 4 & $5d^2$ & SIC ($d^2$ equiangular lines), see \S \ref{SICsection} \\
$d$ & 5 & $6d^2$ & SIC  \\
1 & $t$ & $t+1$ & $(t+1)$-th roots of unity \\
2 & 2 & 6 & ONB \\
2 & 3 & $8^*$ & ONB \\
2 & 4 & $20^*$ & SIC \\
2 & 5 & $24^*$ & SIC  \\
2 & 6 & $42^*$ & MUB (mutually unbiased bases)  \\
2 & 7 & $48^*$ & MUB  \\
2 & 10 & $132$ & $12$-point spherical $(5,5)$-design \\
2 & 11 & $144$ & $12$-point spherical $(5,5)$-design \\
2 & 14 & $360$ & $24$-point spherical $(7,7)$-design \\
2 & 15 & $384$ & $24$-point spherical $(7,7)$-design \\
\hline
\end{tabular}
\end{center}
\end{table}

\setlength{\tabcolsep}{3pt}
\renewcommand{\arraystretch}{1.15}
\begin{table}[H]
\fontsize{10pt}{10pt}\selectfont
\caption{\small The number ${1\over2}\vert\tau_k^S\cap\{(p,q)\}_{p+q=m}\vert$ 
of pairs of nonprojective indices $\{(p,q),(q,p)\}$, $p\ne q$,
$p+q=m$ of order $m$ which are integrated, respectively not integrated by any set 
of unit vectors in $\Cd$ multiplied by the $k$-th roots of unity.}
\begin{center}
\label{taukSTable}
\begin{tabular}{|l| l|p{0.9cm}|p{0.9cm}|p{0.9cm}| p{0.9cm}|p{0.9cm}|p{0.9cm}|p{0.9cm}| p{1.1cm}|}
\hline
index & \multicolumn{9}{l|}{Number index pairs integrated and not integrated 
by a design} \\
degree & \multicolumn{9}{l|}{ multiplied by the $k$-th roots of unity} \\
\cline{2-10}
$m$ & $k=2$ & $k=3$ & $k=4$ & $k=5$ & $k=6$ & $k=7$ & $k=8$ & $k=9$ & $k=10$ \\ \hline
 1 & 1 \ 0 & 1 \ 0 & 1 \ 0 & 1 \ 0 & 1 \ 0 & 1 \ 0 & 1 \ 0 & 1 \ 0 & 1 \ 0 \\
 2 & 0 \ 1 & 1 \ 0 & 1 \ 0 & 1 \ 0 & 1 \ 0 & 1 \ 0 & 1 \ 0 & 1 \ 0 & 1 \ 0 \\
 3 & 2 \ 0 & 1 \ 1 & 2 \ 0 & 2 \ 0 & 2 \ 0 & 2 \ 0 & 2 \ 0 & 2 \ 0 & 2 \ 0 \\
 4 & 0 \ 2 & 2 \ 0 & 1 \ 1 & 2 \ 0 & 2 \ 0 & 2 \ 0 & 2 \ 0 & 2 \ 0 & 2 \ 0 \\
 5 & 3 \ 0 & 2 \ 1 & 3 \ 0 & 2 \ 1 & 3 \ 0 & 3 \ 0 & 3 \ 0 & 3 \ 0 & 3 \ 0 \\
 6 & 0 \ 3 & 2 \ 1 & 2 \ 1 & 3 \ 0 & 2 \ 1 & 3 \ 0 & 3 \ 0 & 3 \ 0 & 3 \ 0 \\
 7 & 4 \ 0 & 3 \ 1 & 4 \ 0 & 3 \ 1 & 4 \ 0 & 3 \ 1 & 4 \ 0 & 4 \ 0 & 4 \ 0 \\
 8 & 0 \ 4 & 3 \ 1 & 2 \ 2 & 4 \ 0 & 3 \ 1 & 4 \ 0 & 3 \ 1 & 4 \ 0 & 4 \ 0 \\
 9 & 5 \ 0 & 3 \ 2 & 5 \ 0 & 4 \ 1 & 5 \ 0 & 4 \ 1 & 5 \ 0 & 4 \ 1 & 5 \ 0 \\
10 & 0 \ 5 & 4 \ 1 & 3 \ 2 & 4 \ 1 & 4 \ 1 & 5 \ 0 & 4 \ 1 & 5 \ 0 & 4 \ 1 \\
11 & 6 \ 0 & 4 \ 2 & 6 \ 0 & 5 \ 1 & 6 \ 0 & 5 \ 1 & 6 \ 0 & 5 \ 1 & 6 \ 0 \\
12 & 0 \ 6 & 4 \ 2 & 3 \ 3 & 5 \ 1 & 4 \ 2 & 6 \ 0 & 5 \ 1 & 6 \ 0 & 5 \ 1 \\
13 & 7 \ 0 & 5 \ 2 & 7 \ 0 & 6 \ 1 & 7 \ 0 & 6 \ 1 & 7 \ 0 & 6 \ 1 & 7 \ 0 \\
14 & 0 \ 7 & 5 \ 2 & 4 \ 3 & 6 \ 1 & 5 \ 2 & 6 \ 1 & 6 \ 1 & 7 \ 0 & 6 \ 1 \\
15 & 8 \ 0 & 5 \ 3 & 8 \ 0 & 6 \ 2 & 8 \ 0 & 7 \ 1 & 8 \ 0 & 7 \ 1 & 8 \ 0 \\
\hline
\end{tabular}
\end{center}
\end{table}

\section{Symmetries of a design}

The symmetry group of a set of points $(v_j)$ in $\Rd$ or $\Cd$, 
and the (projective) symmetry group of the corresponding set of lines
$(c_jv_j)$, $|c_j|=1$ are well studied. These are defined to be the 
group of permutations of the indices of the points/lines which can be
realised by the action of a linear map, and can be calculated (see
\cite{W18}) as the permutations which preserve 
the entries of the Gramian and the $m$-products
$$ \inpro{v_j,v_k}, \qquad
\inpro{v_{j_1},v_{j_2}}\inpro{v_{j_2},v_{j_3}}
\cdots\inpro{v_{j_{m-1}},v_{j_m}}\inpro{v_{j_m},v_{j_1}}, $$
respectively. The relevant observations, for us here, are
 
\begin{itemize}
\item The symmetry group of a design, and the projective symmetry group of a 
projective design can be calculated.
\item If a design is a $G$-orbit, then the action group of $G$ gives a subgroup
of the symmetries of the design.
\item The projective symmetry group is larger than the symmetry group
(in general).
\item Symmetries of a complex projective design are 
inherited by the
corresponding real spherical $t$-designs of Theorem \ref{rootsofunitykdesignTheorem},
and there may be additional real symmetries (see Example \ref{detailedSICsymmex}).
\end{itemize}

We now illustrate these ideas with a detailed example.

\begin{example} 
\label{detailedSICsymmex}
Consider the $4$ equiangular lines in $\CC^2$ (Example \ref{SICexample}).
This SIC $X$ can be constructed as a (projective) orbit of
a (fiducial) vector $v$ under the Weyl-Heisenberg group
(see Section \ref{SICsection}) which is generated by noncommuting matrices $S$ and $\gO$,
where
$$ S=\pmat{0&1\cr1&0}, \quad \gO:=\pmat{1&0\cr0&-1},
\qquad v={1\over\sqrt{6}}\pmat{\sqrt{3+\sqrt{3}}\cr e^{\pi i/4}\sqrt{3-\sqrt{3}}}. $$
The identification of $a+ib\in\CC$ with $(a,b)\in\RR^2$, leads to a corresponding
group action on $\RR^4$, where the matrices and vectors are obtained by replacing
the entry $a+ib$ by 
$$ \pmat{a&-b\cr b&a}, \qquad
\pmat{a\cr b}, $$
respectively, e.g., 
$$ [S]=\pmat{0&0&1&0\cr 0&0&0&1\cr 1&0&0&0\cr 0&1&0&0}, \quad
 [\gO]=\pmat{1&0&0&0\cr 0&1&0&0\cr 0&0&-1&0\cr 0&0&0&-1}, \qquad
[v]={1\over\sqrt{6}}\pmat{\sqrt{3+\sqrt{3}}\cr0\cr{1\over\sqrt{2}}\sqrt{3-\sqrt{3}}\cr {1\over\sqrt{2}}\sqrt{3-\sqrt{3}}}. $$

The orders of the symmetry groups of the spherical $4$-design and 
$5$-design in $\RR^4$ consisting of $20$ and $24$ vectors obtained by 
multiplying $X$ by the fifth and sixth roots of unity, respectively, 
are given in \cite{CHS03} as $15$ and $1152$.
These symmetry groups (for $\CC^2$ or $\RR^4$) contain an element of 
order $5$ and $6$, respectively, corresponding to the symmetry given by multiplying
by the given root of unity. The projective symmetry group 
(whose elements act as matrices, up to a unit scalar multiple) for the SIC
$$ X=\{c_1v, c_2\gO v, c_3Sv, c_4S\gO v\}, \qquad c_j\in\CC, \ |c_j|=1, $$
has order $12$, and is generated by $S$, $\gO$ and an element $Z$ 
of order $3$. A priori, a given choice for $X$ will not inherit any of
the projective symmetries, though it maybe possible to coerce it to, by 
a suitable choice of the scalars $c_j$. Since $S^2=I$, the choice
$$ X_2:=\{v, \gO v, Sv, S\gO v\} $$
has a symmetry of order $2$, given by $S$. The corresponding real and complex
spherical $4$-designs therefore have a symmetry of order at least $10$, and 
direct computation shows this gives all of the symmetry group. A careful calculation
shows that $X$ can be chosen so that it has a symmetry of order $3$, corresponding
to an element $Z$ of order $3$, as follows
$$ X_3=\{v,-i\gO v, iSv,S\gO v\}, 
\qquad Z:={1\over\sqrt{2}}\pmat{\zeta^{17}&\zeta^{17}\cr\zeta^{11}&\zeta^{23}}, \quad
\zeta:=e^{2\pi i/24}, $$
where 
$$ Z^3=I, \qquad
Z v = -i\gO v, \quad Z(-i\gO v)=iSv, \quad Z(iSv)=v,\quad
Z(S\gO v)=S\gO v. $$
For this choice, the corresponding real and complex spherical $4$-designs 
have symmetry groups of order $15$.
Finally, the symmetry groups for the complex/real $5$-designs given
by $X_2$ and $X_3$ are $24,48$ and $72,1152$. The additional symmetries
for the real design can be explained as ``symmetries'' for the complex design
which involve complex conjugation (which is real-linear, but not complex
linear). 
\end{example}

\section{The harmonic Molien series for various groups}
\label{harmMolseriescalcsect}

Recall (Theorem \ref{Gorbitdesigntype}) that the set of indices integrated by every $G$-orbit is 
$$ \tau_G := \{(0,0)\} \cup \{(p,q): H(p,q)^G=0\}. $$
Since a union of $\tau$-designs is a $\tau$-design, this also extends to orbits of more than one vector.
Let $k$ be the order of the group of scalars in the action group of $G$.
In view of Theorem \ref{rootsofunitykdesignTheorem} and Lemma \ref{roleofscalarslemma},
the projective indices
$$ \tau_P=\tau_P^G:=\tau_G\cap \{(0,0),(1,1),(2,2),\ldots\}$$
and the scalar indices
$$ \tau_S=\tau_S^G := \tau_k^S := \{(p,q) : p-q \not\equiv 0 \mod k\}, $$
are of particular importance. This leads to a partition
\begin{equation}
\label{tauGdecomp}
\tau_G = \tau_P \cup \tau_S \cup \tau_E \qquad
\hbox{(disjoint union)},
\end{equation}
where we call $\tau_E=\tau_E^G$ the {\bf exceptional indices} for $G$.

For a general design $X$, we can choose $G$ to be any subgroup of its symmetry group (which could be trivial),
and partition the set of indices it integrates as
$$ \tau_X=\tau_G\cup\tau_E^{X,G} = \tau_P^G \cup \tau_S^G\cup \tau_E^G\cup\tau_E^{X,G}
\qquad \hbox{(disjoint union)}, $$
where $\tau_E^{X,G}$ are the indices $(p,q)$ that are integrated by $X$, but not as a direct consequence 
of the symmetries $G$, i.e., for which $H(p,q)^G\ne0$.

We now calculate the decomposition (\ref{tauGdecomp}) of $\tau_G$ into its projective, scalar and exceptional 
indices for various groups $G$. In this regard, we note that:

\begin{itemize}
\item $\tau_G$ is given by the harmonic Molien series for $G$, which depends
only on $G$ up to similarity, and can be
calculated from its conjugacy classes (Example \ref{harmMolmagmacode}).
\item $\tau_P$ depends only on an associated subgroup 
of $\SU(\Cd)$, defined up to conjugacy, called the
{\bf canonical} ({\bf abstract error}) {\bf group} (for $G$) \cite{CW17}.
This finite group is generated by the matrices of $G$ (or a generating set) 
multiplied by a suitable scalar to have determinant $1$,
together with
the scalar matrices in $\SU(\Cd)$.
\item $\tau_S=\tau_k^S$ is determined by $k$ the order of the subgroup of scalar matrices in $G$,
which is easily calculated, e.g., if $G$ is irreducible then this is the centre of $G$.
\end{itemize}


We now consider the finite irreducible complex reflection groups.
These are well studied (see \cite{LT09}), and have been classified 
by Shephard and Todd into three infinite classes, and $33$ 
exceptional groups with Shephard-Todd numbers $4,5,\ldots,37$.

The following calculations were done in {\tt magma} 
using the code given in Example \ref{harmMolmagmacode}.

\begin{example}
\label{BinTetexample}
 (Binary tetrahedral group)
The canonical group for the Shephard-Todd groups 4,5,6,7
is the binary tetrahedral group of order $24$
$$ G=\inpro{a,b}, \qquad 
 a:= \pmat{i&0\cr0&-i}, \quad
b:= {1\over\sqrt{2}} \pmat{\gep&\gep^3\cr\gep&\gep^7},
\quad \gep=\sqrt{i}={1\over\sqrt{2}}(1+i), $$
for which
$$ \tau_P=\{(0,0),(1,1),(2,2),(5,5)\}, \qquad \tau_S=\tau_2, 
\qquad |\tau_E|=16, $$
where 
\begin{align}
\label{tauEbintet}
\tau_E=\{ 
& (0,2),(0,4),(1,3),(0,10),(1,9),(2,8),(3,7),(4,6), \cr
& (2,0),(4,0),(3,1),(10,0),(9,1),(8,2),(7,3),(6,4) \}. 
\end{align}
Each Shephard-Todd group has the same projective indices, and other ones are
\begin{align*}
\ST(4): \quad 
&\tau_S=\tau_2^S, \quad\ 
\tau_E=\{(0,2),(2,0),(1,5),(5,1),(2,4),(4,2),(2,8),(8,2)\}, \cr
\ST(5): \quad &\tau_S=\tau_6^S, \quad\ \tau_E=\{(2,8),(8,2)\}, \cr
\ST(6): \quad &\tau_S=\tau_4^S, \quad\ \tau_E=\{(1,5),(5,1)\}, \cr
\ST(7): \quad &\tau_S=\tau_{12}^S, \quad \tau_E=\{\}.
\end{align*}
\end{example}

Our calculations of $\tau_P$ and $\tau_E$, which appear to be finite, were done by considering
all indices $(p,q)$ with $p+q\le n$, for $n$ large, e.g., $n=100$. 
This leads to the conjecture:

\begin{conjecture} For every unitary action of a finite group $G$, we have
	\begin{enumerate}[(a)]
\item $\tau_P^G$ and $\tau_E^G$ are finite for all groups $G$.
\item $\tau_E^G$ consists of indices $(p,q)$ with $p+q$ even and
$p+q\le 2M$, $M= \max\{t:(t,t)\in\tau_P^G\}$. 
\end{enumerate}
\end{conjecture}

\noindent
These are supported by all our other calculations.

The binary tetrahedral group and Shephard-Todd group $4$ 
 are isomorphic as abstract groups,
but have different exceptional indices, and hence Molien series.

\begin{example} Different faithful irreducible representations of the same abstract group may 
have different harmonic Molien series, i.e., the harmonic Molien series depends on the 
representation. For example, the binary tetrahedral group and the
Shephard-Todd group number $4$ (of order $24$), 
are isomorphic subgroups of $\GL(\CC^2)$ which have different
harmonic Molien series.
Moreover, there is a third $2$-dimensional faithful irreducible 
representation of this abstract group, which has the same harmonic Molien series 
as the Shephard-Todd group, and so the harmonic Molien series can also be equal.
\end{example}

\begin{example} (Binary octahedral group)
The canonical group for the Shephard-Todd groups
$8,9,10,11,12,13,14,15$ is the binary octahedral group of order $48$
$$ G=\inpro{a,b}, \qquad
a:={1\over2} \pmat{-1-i&1-i\cr-1-i&-1+i}, \quad
b:={1\over\sqrt{2}}\pmat{1-i&0\cr0&1+i},$$
for which 
$$ \tau_P=\{(0,0),(1,1),(2,2),(3,3),(5,5),(7,7),(11,11)\}, \qquad \tau_S=\tau_{2},
\qquad |\tau_E|=58. $$
\end{example}

\begin{example} (Binary icosahedral group)
The canonical group for the Shephard-Todd groups
$16,17,18,19,20,21,22$
is the binary icosahedral group of order $120$
$$ G=\inpro{a,b}, \qquad
a:={1\over2} \pmat{\varphi^{-1}-\varphi i&1\cr-1&\varphi^{-1}+\varphi}, \quad
b:=\pmat{-i&0\cr0&i}, \quad \varphi:={1+\sqrt{5}\over2},$$
for which
$$ \tau_P=\{(p,p): p=0,1,2,3,4,5,7,8,9,11,13,14,17,19,23,29\},
\quad \tau_S=\tau_2, \quad
|\tau_E|=330. $$
The sequence $1,2,3,4,5,7,8,9,11,13,14,17,19,23,29$ 
which determines the projective indices
is sequence A210576 in the on-line encyclopedia of integer sequences \cite{OEIS23},
where it is described as the
``positive integers that cannot be expressed as a sum of one or more nontrivial 
binomial coefficients'', 
with some connections to quantum mechanics and the symmetries of the dodecahedron mentioned.
\end{example}

\begin{example} 
\label{binarydihedexample}
(Binary dihedral group)
The binary dihedral group of order $4m$
$$ G=D_{2m}\inpro{a,b}, \qquad
a:=\pmat{ \go &0\cr 0&\gw^{-1}}, \quad
b:=\pmat{0&-1\cr 1&0}, \quad \go=\go_{2m}=e^{\pi i/m}, $$
for which $|\tau_E|=2\lfloor{m/2}\rfloor^2$ and
$$ \tau_P=\{(0,0),(1,1),(3,3),(5,5),\ldots,(m^*,m^*)\}, \qquad 
m^*:=
\begin{cases}
m-2, & \hbox{$m$ odd}; \cr
m-1, & \hbox{$m$ even}. \cr
\end{cases} $$
This was proved in \cite{M23} by calculating (\ref{HpqGineigenvalues}) explicitly
from the eigenvalues of the group elements.
In this regard, the nondiagonal elements 
$\{b,ab,a^2b,\ldots,a^{2m-1}b\}$
of $D_{2m}$ have the same vector of eigenvalues, i.e., $\gl_g=(i,-i)$.
\end{example}

\begin{example} (Cyclic group)
The cyclic group of order $m$ as a subgroup of $\SU(\CC^2)$
$$ G=\inpro{a}, \qquad
a:=\pmat{ \go &0\cr 0&\gw^{-1}}, 
\quad \go=\go_{m}=e^{2\pi i/m}. $$
for which $|\tau_E|=0$, $\tau_P=\{(0,0)\}$.
\end{example}

The above examples of canonical groups give all the finite subgroups of $\SU(\CC^2)$
(see Theorem 5.14 \cite{LT09}), and so we have the following very detailed description
of the possible projective
indices for groups acting on $\CC^2$.

\begin{conjecture}
Let $G$ be a finite group with a linear action on $\CC^2$.
Then its projective indices $\tau_P^G$ and associated canonical group 
are one of
\begin{enumerate}
\item $\{(0,0)\}$ (cyclic group)
\item $\{(0,0),(1,1),(3,3),\ldots,(m^*,m^*)$ (binary dihedral group of order $4m$)
\item $\{(0,0),(1,1),(2,2),(5,5)\}$ (binary tetrahedral group).
\item $\{(0,0),(1,1),(2,2),(3,3),(5,5),(7,7),(11,11)\}$ (binary octahedral group)
\item $\{(p,p): p=0,1,2,3,4,5,7,8,9,11,13,14,17,19,23,29\}$
(binary icosahedral group).
\end{enumerate}
\end{conjecture}

The projective indices of the action group of $G$ and the associated
canonical group are the same. The above list gives all possible canonical groups,
and so the only part of the conjecture unproved is the assumption that there
are no further projective indices $(p,p)$ with $p>100$ (where we ceased our computation).

In other words:
\begin{itemize}
\item
To obtain high order spherical designs for $\CC^2$ (and presumbably $\CC^d$) from a group action
one must take the orbit of more than one vector (see \cite{MW19}). 
\end{itemize}

Our calculations also support the conjecture: the exceptional indices for a 
canonical group are determined by the projective indices.

\begin{conjecture}
If $G$ is a canonical group of matrices, i.e., a finite subgroup of $\SU(\Cd)$, 
then its exceptional indices are
$$ \tau_E^G = \{(p,q):p+q=2m,(m,m)\in\tau_P^G\}\setminus (\tau_P^G \cup \tau_d^S). $$
In particular, every orbit integrates
$$ \Harm(2m)=H(2m,0)\oplus H(2m-1,1)\oplus\cdots\oplus H(0,2m), \qquad (m,m)\in\tau_P^G. $$
\end{conjecture}

\begin{example} (The Weyl-Heisenberg and Clifford groups)
For the Weyl-Heisenberg group acting on $\Cd$ (see Section \ref{SICsection}), we have
$$ \tau_P=\{(0,0),(1,1)\}, \qquad \tau_S=\tau_{d}, \qquad |\tau_E|=0. $$
The normaliser of the Weyl-Heisenberg group in the unitary matrices is
the Clifford group (see \cite{W18}). For the Clifford group acting on $\Cd$,
we observe 
the following pattern
\begin{align*}
\tau_P &=\{(0,0),(1,1),(2,2),(3,3),(5,5),(7,7),(11,11)\}, \qquad d=2, \cr
\tau_P &=\{(0,0),(1,1),(2,2)\}, \qquad d=3,5,7,\ldots, \cr
\tau_P &=\{(0,0),(1,1)\}, \qquad d=4,6,8\ldots. 
\end{align*}
\end{example}

\begin{example} (Shephard-Todd groups)
The projective indices for the 
Shephard-Todd groups $4,\ldots,37$
were calculated (see Table \ref{STprojindsTable}). For the groups 
 $16,17,\ldots,22$ there is a complete set of projective indices up to
$(9,9)$, except for the ``missing'' index $(6,6)$. 
Thus each orbit is a spherical $(5,5)$-design. In \cite{MW19} a union of two such $(5,5)$-designs 
(with a small number of vectors)
was taken to obtain a $(6,6)$-design. These designs were observed to be
$(9,9)$-designs, which is explained by the projective indices $(7,7),(8,8),(9,9)$.
Thus any orbit which is $(6,6)$-design is automatically a $(9,9)$-design.
There are similar missing indices for other Shephard-Todd groups.
\end{example}




\setlength{\tabcolsep}{3pt}
\renewcommand{\arraystretch}{1.15}
\begin{table}[H]
\fontsize{10pt}{10pt}\selectfont
\caption{\small Projective indices for the exceptional Shephard-Todd reflection groups
acting on $\Cd$.}
\label{STprojindsTable}
\begin{center}
\begin{tabular}{|p{3.5cm}|p{0.8cm}|p{6.9cm}|}
\hline
ST group number & $d$ & nontrivial projective indices $\tau_P^G\setminus\{(0,0)\}$ \\
\hline
 4,5,6,7 & 2 & $(1,1),(2,2),(5,5)$ \cr
 8,9,10,11,12,13,14,15 & 2 & $(1,1),(2,2),(3,3),(5,5),(7,7),(11,11)$ \cr
 16,17,\ldots,22 & 2 & $(t,t)$, where \cr
 & & $t=1,2,3,4,5,7,8,9,11,13,14,17,19,23,29$ \cr
 23 & 3 & $(1,1)$ \cr
 24 & 3 & $(1,1),(2,2)$ \cr
 25,26 & 3 & $(1,1),(2,2)$ \cr
 27 & 3 & $(1,1),(2,2),(3,3)$ \cr
 28 & 4 & $(1,1)$  \cr
 29 & 4 & $(1,1),(2,2)$  \cr
 30 & 4 & $(1,1),(3,3),(5,5)$ \cr
 31 & 4 & $(1,1),(2,2),(3,3),(5,5)$  \cr
 32 & 4 & $(1,1),(2,2),(3,3),(5,5)$  \cr
 33 & 5 & $(1,1),(2,2)$ \cr
 34 & 6 & $(1,1),(2,2),(3,3)$ \cr
 35 & 6 & $(1,1)$ \cr
 36 & 7 & $(1,1)$ \cr
 37 & 8 & $(1,1),(3,3)$ \cr
\hline
\end{tabular}
\end{center}
\end{table}

\section{Weyl-Heisenberg SICS}
\label{SICsection}

Let $G$ be the Weyl-Heisenberg group acting on $\Cd\cong\CC^{\ZZ_d}$, $d\ge2$
(see \cite{W18}).
This is generated by the shift and modulation operators $S$ and $\gO$,
which 
are given by
$$ Se_j:=e_{j+1}, \qquad \gO e_j :=\go^j e_j, \quad \go:=e^{2\pi i\over d}, \qquad j\in\ZZ_d. $$
This group has $d^3$ elements 
$\{\go^\ell S^j\gO^k\}_{\ell,j,k\in\ZZ_d}$,
including the $d$ scalar matrices $\{\go^\ell I\}_{\ell\in\ZZ_d}$.
The set of $d^2$ lines given by the $G$-orbit of a unit vector $v\in\Cd$ 
is said to be a ({\bf Weyl-Heisenberg}) {\bf SIC} if the lines are equiangular,
and such a set is a spherical $(2,2)$-design.

Here we show that the condition (\ref{G-orbitintegrates}) of Lemma \ref{keylemma} 
for a $G$-orbit
to be a spherical $(2,2)$-design naturally leads to the standard
equations for it to be a SIC. In particular, 

\begin{itemize}
\item
A Weyl-Heisenberg orbit is a SIC if and only if it is a spherical 
$(2,2)$-design.
\end{itemize}

We recall that a spherical $(2,2)$-design is one which integrates the
indices $(1,1),(2,2)$. Since the action of the Weyl-Heisenberg group $G$
on $\Cd$ is irreducible, by Example \ref{H(1,1)irreduciblecdn}  
each orbit integrates $(1,1)$. This can also be verified by 
calculating (\ref{SerreCondition}) directly
$$ \sum_{g\in G} \tr(g)\tr(g^{-1}) = \sum_{j=0}^d \tr(w^j I)\tr(w^{-j}I)
= d^3=|G|. $$
Therefore, we need only investigate when an orbit integrates the index $(2,2)$.



\begin{lemma} 
\label{dim(H(2,2)^G)lemma}
Let $G=\inpro{S,\gO}$ be the Heisenberg group acting
on $\Cd$, $d\ge2$. Then
$$ \dim(H(2,2))={1\over4}d^2(d+3)(d-1), \quad 
\dim(H(2,2)^G)=\begin{cases} {1\over4}d(d+2), & \hbox{$d$ even}; \cr 
{1\over4}(d-1)(d+3), & \hbox{$d$ odd}. \end{cases}  $$
\end{lemma}

\begin{proof} From Proposition \ref{H(p,q)^Gdimensionformula}, we obtain
the general formula
\begin{align*}
\dim(H(2,2)^G)
&= {1\over4}{1\over|G|}\sum_{g\in G} \bigl\{ 
\tr(g)^2\tr(g^{-1})^2+\tr(g)^2\tr(g^{-2})
+\tr(g^2)\tr(g^{-1})^2 \cr
& \qquad\qquad\qquad +\tr(g^2)\tr(g^{-2})
-4\tr(g)\tr(g^{-1}) \bigr\}.
\end{align*}

In our particular case, it suffices to sum over the elements 
$S^j\gO^k$ (and multiply by $d$), since the terms for $g$ and $\go^\ell g$ are 
the same. For $d$ odd, the only such matrix for which the traces in 
the above formula
are not all zero is the identity $I$, which gives
$$ \dim(H(2,2)^G)={1\over4}{1\over d^3} d\bigl\{ d^4+d^3+d^3+d^2-4d^2\bigr\}
={1\over4}(d-1)(d+3). $$
For $d$ even, in addition to $I$, 
there is a contribution to the sum from 
the three matrices $g=S^{d/2},\gO^{d/2},S^{d/2}\gO^{d/2}$, 
which have $\tr(g)=\tr(g^{-1})=0$ and $g^2=g^{-2}=\pm I$, giving
$$ \dim(H(2,2)^G) = {1\over4}(d-1)(d+3) + {1\over4}{1\over d^3}d\bigl\{3d^2\bigr\}
= {1\over4}d(d+2). $$
Finally, the formula for $\dim(H(2,2))$ is given in Example \ref{Hpqdimremark},
or by taking $G=1$ in the general formula for $\dim(H(2,2)^G)$ above.
\end{proof}


Define polynomials $f_{st}\in\Hom(2,2)^G$, by
$$ f_{st}(z):=\sum_{r\in\ZZ_d} z_r\overline{z}_{r+s}\overline{z}_{r+t}z_{r+s+t},
\qquad s,t\in\ZZ_d. $$
Let $\Delta=\sum_j\partial_j\overline{\partial}_j$ be the Laplacian.
A simple calculation gives
\begin{align*}
\Delta(f_{st})(z) 
&=\sum_r \bigl( \gd_{s,0}(\overline{z}_{r+t}z_{r+t}+z_r\overline{z}_{r+s})
+\gd_{t,0}(\overline{z}_{r+s}z_{r+s}+z_r\overline{z}_{r+t}) \bigl) \cr
& =\begin{cases}
0, & s,t\ne0; \cr
2\norm{z}^2, & s\ne0,t=0,\ s=0,t\ne0; \cr
4\norm{z}^2, & (s,t)=(0,0),
\end{cases}
\end{align*}
so that the polynomials 
$\{f_{st}\}_{s,t\ne0}$ and $\{f_{00}-2f_{s0}\}_{s\ne0}$ belong to $H(2,2)^G$.

\begin{lemma}
\label{H(2,2)basislemma}
Let $G$ be the Weyl-Heisenberg group. Then a basis for $H(2,2)^G$ is given by
$$ \{f_{st}\}_{s,t\ne0}\cup\{f_{00}-2f_{s0}\}_{s\ne0}. $$
\end{lemma}

\begin{proof}
The polynomials $f_{st}$ satisfy
\begin{equation}
\label{fstsymmetry}
f_{st}=f_{ts}=f_{-s,-t}=f_{-t,-s}, 
\end{equation}
and so there are on the order of $d^2/4$ of them. A careful count gives
$$ |\{f_{st}\}_{s,t\ne0}| 
=\begin{cases}
{1\over4}d^2, & \hbox{$d$ even}; \cr
{1\over4}(d^2-1), & \hbox{$d$ odd},
\end{cases} \qquad
|\{f_{s0}\}_{s\ne0}| 
=\begin{cases}
{d\over2}, & \hbox{$d$ even}; \cr
{d-1\over2}, & \hbox{$d$ odd}.
\end{cases} $$
Hence, by Lemma \ref{dim(H(2,2)^G)lemma}, we have
$$ |\{f_{st}\}_{s,t\ne0}\cup\{f_{00}-2f_{s0}\}_{s\ne0}|
= |\{f_{st}\}_{s,t\ne0}|+|\{f_{s0}\}_{s\ne0}| = \dim(H(2,2)^G). $$
The polynomials in the asserted basis for $H(2,2)^G$ are easily verified to be 
linearly independent, and so, by a dimension count, they are indeed a basis.
\end{proof}

In view of Lemma \ref{keylemma} and Lemma \ref{H(2,2)basislemma}, 
a necessary and sufficient condition
for the $G$-orbit of a unit vector $z\in\CC$ to be a spherical 
$(2,2)$-design is that it satisfies  
\begin{equation}
\label{SIChomocdns}
f_{st}(z)=0, \quad s,t\ne0, \qquad
f_{00}(z)-2f_{s0}(z)=0, \quad s\ne0. 
\end{equation}
The standard conditions for $z$ to give a Weyl-Heisenberg SIC
(see \cite{BW07}, \cite{K08}, \cite{ADF07}) are
\begin{equation}
\label{SICstandardcdns}
f_{st}(z)=0, \quad s,t\ne0, \qquad
f_{s0}(z)={1\over d+1} \quad s\ne0, \qquad
f_{00}(z)={2\over d+1}.
\end{equation}
Clearly, the standard conditions 
(\ref{SICstandardcdns}) imply (\ref{SIChomocdns}).
We now show they are equivalent. Suppose that
(\ref{SIChomocdns}) holds, and that
$\norm{z}^2 
=\sum_j r_j^2=1$, where $r_j:=|z_j|$. Then
$$ 0=\Bigl(\sum_j r_j^2-1\Bigr)^2
= \sum_j r_j^4 +\sum_{s\ne0}\sum_j r_j^2r_{j+s}^2 - 2\sum_j r_j^2+1 
= \sum_j r_j^4 +\sum_{s\ne0}\sum_j r_j^2r_{j+s}^2 -1, $$
Since 
$$ f_{s0}(z)=\sum_j r_j^2 r_{j+s}^2, $$
this gives
$$ \sum_j f_{00}(z) +\sum_{s\ne0} f_{s0}(z) = 1. $$
Thus, 
by the second equation in (\ref{SIChomocdns}),
we obtain
$$ f_{00}(z) +\sum_{s\ne0} f_{s0}(z)
=  f_{00}(z) -{1\over2}\sum_{s\ne0} \bigl(f_{00}(z)-2f_{s0}(z)\bigr)+{1\over2}(d-1)f_{00}(z)
= {d+1\over2}f_{00}(z)=1, $$
which gives
$$ f_{00}(z)={2\over d+1}, \qquad f_{s0}(z) 
=  {1\over2}f_{00}(z) ={1\over d+1}, \quad s\ne0. $$
This establishes the equivalence.

A simple calculation gives
$$ \overline{f_{st}(z)} = f_{-s,t}(z), $$
so that 
\begin{equation}
\label{fstconjsymm}
f_{st}(z) =0 \Iff \overline{f_{st}(z)}=0 \iff f_{-s,t}(z)=0.
\end{equation}
Thus, not all of the equations in (\ref{SIChomocdns}) and (\ref{SICstandardcdns})
are required.
It follows from (\ref{fstsymmetry}) and (\ref{fstconjsymm}), 
that it is sufficient to choose one equation $f_{st}(z)=0$, $s,t\ne0$,
corresponding to each equivalence class of indices 
\begin{equation}
\label{fstequivclassesofindices}
\{ (s,t),(t,s), (-s,t),(-t,s), (-s,t),(-t,s), (s,t),(t,s) \}.
\end{equation}
This reduces the number of equations
$f_{st}(z)=0$, $s,t\ne0$, required to
$$ {1\over2}m_d(m_d+1)
=\begin{cases}
{1\over8}d(d+2), & \hbox{$d$ even}; \cr
{1\over8}(d^2-1), & \hbox{$d$ odd},
\end{cases}
$$
where
$$ m_d:=|\{f_{s0}\}_{s\ne0}| 
=\begin{cases}
{d\over2}, & \hbox{$d$ even}; \cr
{d-1\over2}, & \hbox{$d$ odd}.
\end{cases}
$$
Thus, by counting, we have:

\begin{proposition}
The number of equations from (\ref{SICstandardcdns}), or (\ref{SIChomocdns})
together with $\norm{z}=1$, required to define a spherical $(2,2)$-design
(or SIC) is 
$$ {1\over 2}(m_d+1)(m_d+2)
=\begin{cases}
{1\over8}(d+2)(d+4), & \hbox{$d$ even}; \cr
{1\over8}(d+1)(d+3), & \hbox{$d$ odd}.
\end{cases} $$
A suitable selection is given by taking one polynomial $f_{st}(z)=0$, for
each equivalence class (\ref{fstequivclassesofindices}) of indices.
\end{proposition}

The above count for $d$ odd was given in \cite{BW07}.

\section{Conclusion}

The theory developed here (and in \cite{RS14}) naturally extends to {\bf weighted} spherical designs, 
where (\ref{cuberule}) is replaced by  
\begin{equation}
\label{weightedcuberule}
\int_\SS p(x)\, d\gs(x) = {1\over |X|} \sum_{x\in X} w_x p(x),
\end{equation}
for ``weights'' $w_x\in\RR$, with $\sum_{x\in X} w_x=|X|$.
These are well suited to the construction of high order designs with symmetry, where
the design is 
a union of orbits \cite{MW19}.

We have primarily concentrated on the algebraic aspects of the theory and their implementation
to construct real and complex spherical designs. Also of interest is a variational characterisation of
complex designs similar to (\ref{varcharofdesigns}) and estimates on the minimal
number of points a given class of designs may have (see \cite{RS14}, \cite{W20}).

Finally, here we studied the class of complex designs 
up to unitary equivalence, by
considering the absolutely irreducible subspaces of homogeneous harmonic polynomials
that they integrate. We pointed out the analogous development for real designs, which is
simpler, as there are fewer invariant subspaces. A corresponding, more complicated, 
development for quaternionic designs could similarly be developed.

\bibliographystyle{alpha}
\bibliography{references}
\nocite{*}

\vfil
\end{document}